\documentclass[12pt,reqno]{amsart}
\input{package.sty}
\input{mathcommand.sty}

\begin{document}

\title{Convergence of interfaces in boundary reactions}
\author{Aditya Kumar}
\address{Department of Mathematics, Johns Hopkins University, 3400 N. Charles Street, Baltimore, MD 21218 , USA}
\email{akumar65@jhu.edu}

\begin{abstract}
In this paper we study the singular limit for critical points of boundary reactions 
\begin{equation*}
    (-\Delta)^{\frac{1}{2}}u = \frac{1}{\varepsilon}(u-u^3) \quad \text{in } U \subset \textbf{R}^n .
\end{equation*}
 We show the existence of a $(n-1)$-rectifiable energy concentration set. Furthermore, we show that the limit of the energy measures can be associated to a stationary, $(n-1)$-rectifiable varifold supported in the concentration set. This is analogous to a result of Hutchinson and Tonegawa for phase transitions.
\end{abstract}

\maketitle

\section{Introduction}

\subsection{Minimal surfaces and phase transitions} Minimal hypersurfaces are the critical points of the area functional. The relationship between minimal hypersurfaces and interfaces formed by double well phase transitions has been a subject of many works over the past few decades \cite{g,gg,ht,m,st,tw}. The latter are the critical points of the Allen-Cahn energy functional
\begin{equation} \label{ace}
    \frac{1}{2}\int_U |\nabla u|^2 + \frac{1}{\ve ^2} W(u) \,dx
\end{equation}
whose Euler-Lagrange equation is the elliptic Allen-Cahn equation 
\begin{equation} \label{ac}
    -\Delta u_{\ve} = -\frac{1}{\ve ^2} W'(u_{\ve})  \quad \text{in } U 
\end{equation}
here $u$ is the density, $U$ is a bounded domain, $W: \RR \to [0,\infty)$ is a double well potential with two minima $1$ and $-1$, which are the densities of the stable phases. In this paper $W(t)=\frac{(1-t^2)^2}{4}$. The potential term $W$ penalizes densities away from $\pm 1$, and the Sobolev energy term, $|\nabla u|^2$ penalizes oscillations, so a critical point usually has two regions where $u$ is $1$ and $-1$ respectively with a diffuse phase transition interface which can be seen as an $\ve$-neighbourhood of $u_{\ve}^{-1}(0)$-the zero level set. As $\ve \to 0$, it was expected that the phase transition interface will converge to a minimal hypersurface \cite{i}. We describe some of the results in this direction.

If $u_\ve$ are minimizers then as $\ve \to 0$, after passing to a subsequence, $u_\ve \to \chi_A - \chi_{A^c}$. Here the boundary $\partial A = \Sigma$ is an area minimizing hypersurface. This was shown in \cite{mm} by Modica and Mortola. Further, in \cite{cc}, Caffarelli and Cordoba proved uniform convergence of the level sets 
    $\{|\ue| \leq \delta\}$ to the minimal surface.

For critical points with bounded energy, but without any stability or minimizing condition, Hutchinson and Tonegawa \cite{ht} proved that as $\ve \to 0$ the energy measures converge to a stationary integer rectifiable $(n-1)$-varifold, i.e a generalized minimal hypersurface. Later in \cite{tw}, Tonegawa and Wickramesekara proved that if $\{\ue\}$ are stable critical points, then this limiting varifold is a stable minimal hypersurface.

This relationship between minimal surfaces and the Allen-Cahn equation has been used in the last decade in several advances in the theory minimal surfaces. We mention two: in \cite{g}, building on the works \cite{ht,tw} Guaraco gave a min-max construction of minimal surfaces based on Allen-Cahn equation as an alternative to the Almgren-Pitts theory \cite{p}. Chodosh and Mantoulidis proved the multiplicity one conjecture for minimal surfaces in $3$-manifolds in \cite{cm}.  

\subsection{Boundary reactions and nonlocal phase transitions} A natural variant of energy (\ref{ac}) analysed by Alberti, Bouchitté and Seppecher in \cite{abs}, and then extensively studied by Cabré and Solà-Morales in \cite{csm}, has the potential $W$ placed on the boundary of the euclidean half space, i.e. they consider the energy
\begin{equation} \label{hace}
   E_{\ve}(u)= \frac{1}{2}\int_{\RR^{n+1}_+} |\nabla u|^2 \,dx + \frac{1}{\ve}\int_{\partial \RR^{n+1}_+}  W(u) \mh
\end{equation}
The critical points of (\ref{hace}) satisfy the nonlinear Neumann problem 
 \begin{equation} \label{ext}
        \begin{cases}
            \Delta \ue = 0 \qquad &\text{in } \RR^{n+1}_+ \\
            \pdv{\ue}{\nu} = -\frac{1}{\ve}W'(u_\ve) \qquad  &\text{on  } \partial \RR^{n+1}_+ 
        \end{cases}
    \end{equation}
This equation appears in crystal dislocation \cite{gm,to}. It also occurs in analysis of vortices for soft thin films \cite{k}. This problem also has a nonlocal formulation on $\RR^n \cong \partial \RR^{n+1}_+$ as a fractional Allen-Cahn equation \cite{cs} 
\begin{equation}\label{hac}
    (-\Delta)^{1/2} u_{\ve} = -\frac{1}{\ve} W'(u_{\ve})  \quad \text{in } \RR^n
\end{equation}
where 
$$(-\Delta)^{1/2}u(x) = 2 {\, \rm p.v.} \int_{\RR^n} \frac{u(x)-u(y)}{|x-y|^{n+1}} \,dy$$

Seen this way, the problem appears when the effects of long-range interactions in phase transitions are studied through a nonlocal energy \cite{msw, sv,sv1,svp}. In energy functional (\ref{ace}) the $H^1$ Sobolev energy term $\int |\nabla u|^2 \,dx=\norm{u}_{H^1}^2$ is replaced by contribution from $U$ in the nonlocal energy $\norm{u}_{H^s}^2$ for $s\in(0,1)$, giving the energy functional 
\begin{equation*}
    \norm{u}_{H^s(U)}^2  + \frac{1}{\ve ^{2s}}\int_U  W(u) \,dx
\end{equation*}
The $H^s$ energy of $u$ is 
$$\norm{u}_{H^s}^2 =\int \int _{\RR^d \times \RR^d} \frac{|u(x)-u(y)|^2}{|x-y|^{n+2s}} \,dx\,dy  $$
and therefore the appropriate energy functional with $H^s$, instead of $H^1$ energy in (\ref{ace}),  is 
\begin{equation} \label{nace}
    \frac{1}{2}\int \int _{\RR^d \times \RR^d \backslash U^c \times U^c} \frac{|u(x)-u(y)|^2}{|x-y|^{n+2s}} \,dx\,dy  + \frac{1}{\ve ^{2s}}\int_U  W(u) \,dx
\end{equation}
 Note that here the boundary value of $u$ will be fixed in complement of $U$ and $U^c=\RR^n \backslash U$. The Euler-Lagrange equation is the fractional Allen-Cahn equation, i.e, equation (\ref{ac}) with the laplacian $-\Delta$ replaced by the fractional laplacian $(-\Delta)^s$, for $s\in (0,1)$
\begin{equation}\label{nacs}
    (-\Delta)^{s} u_{\ve} = -\frac{1}{\ve ^{2s}} W'(u_{\ve})  \quad \text{in } U, \quad s\in(0,1)
\end{equation}
with some boundary data in $U^c$. Thus one sees that the equation (\ref{hac}) is just the $s=1/2$ case in above, where 
$$(-\Delta)^{s}u(x)= c_{n,s}\,{\rm p.v.} \int_{\RR^n} \frac{u(x)-u(y)}{|x-y|^{n+2s}} \quad $$

\subsection{Minimal surfaces and boundary reactions} Given the geometric properties of the interface formed by the Allen-Cahn equation as highlighted above, it is of interest to understand the limit interface formed for nonlocal phase transitions as $\ve \to 0$. Here, there are two different outcomes depending on whether $s\in(0,1/2)$ or $s\in [1/2,1).$ The dichotomy is primarily there because the characteristic function of smooth sets are not in $H^s$ when $s\in[1/2,1)$. Whereas for $s\in (0,1/2)$ this is not the case, so as $\ve \to 0$, the $H^s$ energy remains uniformly bounded.

In case of minimizers, Savin and Valdinoci \cite{sv} proved that as $\ve \to 0$, after passing to a subsequence, $u_\ve \to \chi_A - \chi_{A^c}$. If $s\in(0,1/2)$, then the boundary of the set $A$, $\partial A = \Sigma$ is a minimizing  $H^s$-nonlocal minimal hypersurface, i.e. boundaries of sets minimizing the $H^s$ energy. They were first studied by Caffarelli, Roquejoffre and Savin in \cite{crs}.
    
Interestingly if $s\in[1/2,1)$, the limit object is same as in the classical case, i.e, the boundary of the set $A$, $\partial A = \Sigma$ is an area minimizing hypersurface. Analogous to \cite{cc} in the local case, Savin and Valdinoci also prove in \cite{sv1} the uniform convergence of level sets of $\ue$ to the non-local $(s<1/2)$  or local minimal surface $(s\geq1/2)$.
    
For critical points with bounded energy, but without the minimizing condition, only the case with $s\in (0,1/2)$ is known. Millot, Sire and Wang \cite{msw} proved that that the energies converge to a  stationary $H^s$ nonlocal minimal surface. 

One might hope that for $s \in [1/2,1)$, the energy measures of critical points with bounded energy will converge as in the classical case \cite{ht} to a stationary integer rectifiable $(n-1)$-varifold. However a monotonicity formula for the nonlocal energy (\ref{nace}) is known only for $s\in(0,1/2)$. See also the discussion in \cite{csv}. Still for $s=1/2$ we are able to show the convergence to a stationary rectifiable $(n-1)$ varifold. For this we view $u_{\ve}$ as a solution of the non-linear Neumann problem (\ref{ext}) instead of viewing $u_\ve$ as solutions to (\ref{hac}). We state the result below precisely for the half ball $\ball 1 = B_1(0) \cap \RR^{n+1}_+$, and will later specify the exact definition of admissible open sets in $\RR^{n+1}_+$.
\begin{theo} \label{maintheorem}
Let $\{u_\ve \}$ be critical points for the energy functional $E_\ve$ (\ref{hace}), in $\ball 1$ satisfying a uniform energy bound $E_{\ve}(u_\ve) \leq E_o < \infty$.  Then as $\ve \to 0$, there is a naturally associated stationary rectifiable $(n-1)$-varifold $V$, supported on the energy concentration set $\Sigma$ that forms on the boundary $\ball 1 \cap \{x_{n+1}=0 \}$.
\end{theo}

We also remark here that a major motivation for our work came from the very interesting work of Figalli and Serra \cite{fs}, in which they proved that every bounded stable solution of the equation  (\ref{hac}) in $\RR^3$ is one dimensional. Theorem \ref{maintheorem} is the first step in the direction of applying the results of \cite{fs} to the study of minimal hypersurfaces. 

We plan to address the integrality of the limiting varifold and the counterpart of \cite{tw} in this setting in future works.

\subsection{Organization} 
  
  In section $2$, we establish two preliminary results. First, the monotonicity formula for the energy functional (\ref{hace}); to the best of our knowledge it was first given in \cite{ms}. Here we provide a proof based on the proof of \cite{ht} in the classical case. The other result we prove is a technical convergence lemma for solutions $\{ \ue \}$ under a uniform gradient bound.

  In section $3$ we study the behavior of equation (\ref{ext}) in a domain $U$, in the small energy regime. Here we first prove a clearing-out type result on the boundary. Further, we establish an epsilon regularity result for our equation. The epsilon regularity used in conjunction with the convergence lemma is a key tool for many of our proofs. 

  In section $4$, we use the results proved in the previous two sections to construct the energy concentration set $\Sigma$, and prove that it has Hausdorff dimension  $(n-1)$. Further we show that $\Sigma$ is the obstruction to lack of compactness of $\{ \ue\}$ in $H^1$. This is characterised by a defect measure $\mu_{\Sigma}$ supported in the concentration set. We also show that $\mu_{\Sigma}$ is absolutely continuous with respect to $\mathcal{H}^{n-1}\measurer \Sigma$.

  Finally in section $5$, we briefly recall the formalism of varifolds and generalized varifolds. Then we show that the stress-energy tensors associated to $\ue$ are naturally seen as generalized varifolds $V_{\ve}$ that converge as $\ve \to 0$ to a limiting generalized varifold $V$. Further, there is a stationary rectifiable $(n-1)$-varifold $V_{\Sigma}$ that is supported on $\Sigma$ that is naturally associated to $V$.

\subsection*{Notation of sets and boundaries}
\begin{itemize}
    \item  $\RR^n$ is identified with $\partial \RR^{n+1}_+ = \RR^n \times \{0\}$.
    \item  $B_r(x)$ is the open ball in  $\RR^{n+1}$ centered at $x$. If the center is not specified then it is the origin.
    \item $\ball r (x) = B_r(x) \cap \RR^{n+1}_+ $ is a half ball with center $x \in \RR^n$.
    \item $D_r(x)= B_r(x) \cap \RR^{n}$ is a disc on the boundary.
    \item For a set $U$, $\quad U^+=U \cap \RR^{n+1}_+$, $\quad \partial^+ U=\partial  U \cap \RR^{n+1}_+$.
    \item $\partial^0U = \{ x\in U \cap \RR^n \, : \ball r (x) \subset U \text{ for some } r>0 \}$. For example $\partial^0 \ball r (x) = D_r(x)$.
    \item An admissible open set $U$ is a bounded open set in $\RR^{n+1}_+$ such that $\partial U$ is Lipschitz, $\partial^0 U$ is non empty and has a Lipschitz boundary, and $\partial U= \partial^+U \cup \overline{\partial^0U}$
\end{itemize}

\subsection*{Acknowledgement} The author is immensely grateful to his advisor Professor Yannick Sire for his constant encouragement and invaluable guidance. He also wishes to thank Junfu Yao, Yifu Zhou and Jonah Duncan for several helpful discussions.

\section{Monotonicity formula and convergence lemma}
In this section and the next we will be studying the $\ve$-perturbed version of the following problem for $f=-W'$.
\begin{equation} \label{eqn:main1}
        \begin{cases}
            \Delta u = 0 \qquad &\text{in } B_R^+ \\
            \pdv{u}{\nu} = f(u) \qquad  &\text{on  } D_R
        \end{cases}
\end{equation}
A simple proof for the regularity of solutions of this equation can be found in \cite{csm}(lemma 2.3). We state it here for convenience as we will be invoking it a few times. 
\begin{theo}(from \cite{csm})
    Let $R>0$, $\alpha \in (0,1)$ and $u \in L^{\infty}(B^+_{R} \cap H^1(B^+_R)$ be a weak solution of (\ref{eqn:main1}). If $f \in C^{1,\alpha}$, then $u \in C^{2,\alpha}(\overline{B^+_{R/4}})$ and 
    \begin{equation} \label{csmest}
        \norm{u}_{C^{2,\alpha}(\overline{B^+_{R/4}})} \leq C(n,\alpha,R,\norm{f}_{C^{1,\alpha}},\norm{u}_{L^{\infty}_{B_R^+}})
    \end{equation}
\end{theo}
\begin{coro}
    When $f(u)= \frac{1}{\ve}(u-u^3)$, then $u \in C^{\infty}(\ball {R/4} \cup D_{R/4})$.
\end{coro}

\subsection{Monotonicity formula}We now prove a monotonicity formula for the associated energy and a convergence lemma that will be required in the proof of the epsilon regularity result. A monotonicity formula in this setting was proved in \cite{ms}. Here we give a proof following the allen-cahn case \cite{ht}.

 For any $x \in D_R$ and for $0<r\leq R-|x|$, we denote the scaled energy by 
 $$
    I_{\ve}(r,x) =  \frac{1}{r^{n-1}}E_{\ve}(u, \ball {r}(x))
 $$
\begin{theo}
Let $u_\ve \in C^2(\ball R \cup D_R)$, satisfying $|u_\ve|\leq 1$, with $\ve <R$ be a solution of 

    \begin{equation} \label{eqn:main}
        \begin{cases}
            \Delta \ue = 0 \qquad &\text{in } B_R^+ \\
            \pdv{\ue}{\nu} = -\frac{1}{\ve}W'(u_\ve) \qquad  &\text{on  } D_R
        \end{cases}
    \end{equation}
Then for any $x \in D_R$, and for $0< r_1 \leq r_2 \leq R-|x|$, the scaled energy $I_\ve(r,x)$ is monotonically increasing in $r$ and satisfies the following relation
\begin{multline} \label{eqn:mf1}
   I_{\ve}(r_2,x) - I_{\ve}(r_1,x)
    = \int_{r_1}^{r_2} \frac{1}{r^{n-1}}\bigg(\int_{\partial ^+\ball r(x)} |(y-x)\cdot \nabla u|^2 \mh_y \bigg) \,dr \\
    + \int_{r_1}^{r_2} \frac{1}{r^{n}} \bigg( \int_{D_r(x)} \frac{W(u)}{\ve} \mh \bigg) \,dr
    \geq 0
\end{multline}
\end{theo}
\begin{proof}
    
Given any vector field $X \in C_c^1(\ball R \cup D_R)$ such that $X^{n+1}|_{D_R}=0$, we claim that 
\begin{equation} \label{eqn:fv}
    \int_{\ball R} \bigg( \frac{|\nabla u|^2}{2} \D X - DX \langle \nabla u, \nabla u \rangle \bigg)\,dx =- \frac{1}{\ve}\int_{D_R} W(u) \D_{\RR^n} X \mh 
\end{equation}

we multiply the equation (\ref{eqn:main}) by $\nabla u \cdot X$ and integrate by parts. 

\begin{align}
    0 &= \int_{\ball R} \Delta u \nabla u 
    \cdot X \,dx \\
    &= -\int_{\ball R} \nabla[\nabla u \cdot X] \cdot \nabla u \,dx + \int_{D_R} (\nabla u \cdot X)\frac{\partial u}{\partial \nu} \mh \label{ibp}  \\
    &= -\int_{\ball R} \nabla[\nabla u \cdot X] \cdot \nabla u \,dx + \int_{D_R} (\nabla u \cdot X)\frac{-W'(u)}{\ve} \mh \\
    &= \Bigg[-\int_{\ball R} \nabla\frac{|\nabla u|^2}{2} \cdot X \,dx -\int_{\ball R} DX \langle \nabla u, \nabla u \rangle\,dx \Bigg] + \Bigg[\int_{D_R} -\frac{\nabla W(u)}{\ve} \cdot X \mh \Bigg] \\
    &= \int_{\ball R} \bigg( \frac{|\nabla u|^2}{2} \D X - DX \langle \nabla u, \nabla u \rangle\bigg)\,dx + \frac{1}{\ve}\int_{D_R} W(u) \D_{\RR^n} X \mh  
\end{align}

Let $\rho$ be a smooth mollification of $\chi_{\overline{\ball r}}$, the characteristic function of the half ball. Then we take $X(y) = y\rho(|y|)=(y_1\rho(r), \cdots,y_{n+1}\rho(r))$. Then we have 
$$ X^i_j = \delta^i_j \rho + \frac{y_i y_j}{r}\rho '(r)$$ Plugging this into the equation we get 

\begin{equation}
    \int_{\ball R} \bigg( \frac{|\nabla u|^2}{2}[(n-1)\rho+r\rho ')] - \frac{\rho '}{r}(y\cdot \nabla u)^2\bigg)\,dx + \frac{1}{\ve}\int_{D_R} W(u)(n\rho + r \rho ') \mh = 0
\end{equation}

As $\rho \to \chi_{\overline{\ball r}}$, and then dividing by $r^{n-1}$ we get
\begin{equation} \label{eqn:mf}
    \frac{d}{dr} \frac{1}{r^{n-1}} E_{\ve}(u, \ball r) = \frac{1}{r^{n-1}} \int_{\partial ^+ \ball r} (y\cdot \nabla u)^2 \mh + \frac{1}{t^n} \int_{D_r}\frac{W(u)}{\ve} \mh 
\end{equation}

Integrating from $r_1$ to $r_2$ we get (\ref{eqn:mf1})
\end{proof}

The following corollary is an easy consequence of the monotonicity formula (\ref{eqn:mf1}) that will be used frequently.

\begin{coro}
If $I_{\ve}(R,0) \leq \eta$, then for any $x \in D_{R/2}$ and for $0\leq r \leq R-|x|$,  $I_\ve (r,x) \leq 2^{n-1} \eta$
\end{coro}
\begin{proof}
  We have
\begin{equation}\label{eqn:mflemma}
    \begin{split}
                   I_\ve (r,x) &\leq I_\ve (R-|x|,x) \\
                    &\leq \bigg(\frac{1}{R-|x|}\bigg)^{n-1} E_{\ve}(u,\ball R) \\
                    &\leq \bigg(\frac{R}{R-|x|}\bigg)^{n-1} I_\ve (R,0) \leq 2^{n-1} \eta \quad \because |x|<R/2
    \end{split}
\end{equation}  
\end{proof}

\subsection{Convergence lemma}We now prove a convergence lemma that is central to the proof of our epsilon regularity result. It also provides context for the discussion in section 4.

\begin{lemm}\label{lem:conv}
    Let $\{\ui\}_{i \in \NN } \in C^2(B_1\cup D_1)$ be solutions of 
    \begin{equation} \label{eqn:conveqn}
        \begin{cases}
            \Delta \ui = 0 \qquad &\text{in } B_1^+ \\
            \pdv{\ui}{\nu} = \frac{1}{\ve_i}(\ui-\ui^3) \qquad  &\text{on  } D_1
        \end{cases}
    \end{equation}
satisfying
    \begin{equation} \label{eqn:bdry}
       |\ui| \leq 1, |\nabla \ui| \leq 2 \quad \text{in } B_1^+ \cup D_1 \quad \text{and }|\ui|\geq 1/2 \quad \text{on } D_1 
    \end{equation}
Then there is a function $u_{*}$ such that as $\ve_i \to 0$, $\ui \to u_{*}$ in $C^{1,\alpha}_{loc}(\ball 1 \cup D_1)$, and
\begin{equation}
        \begin{cases}
            \Delta u_{*} = 0 \qquad &\text{in } B_1^+ \\
            u_{*} = \pm 1 \qquad  &\text{on  } D_1
        \end{cases}
    \end{equation}
\end{lemm}

\begin{proof}
Note that due to the uniform bound on gradient and $\ui$, there is a weak limit $u_*$ of $\ui$ in $H^1(\ball 1)$. Since, $\ui$ are harmonic, so is $u_*$, and thus $\ui \to u$ in $C^{\infty}_{loc}(\ball 1)$. Further, note that using the uniform $C^0$ and $C^1$ estimates (\ref{eqn:bdry}), interpolation of Holder norms gives 
\begin{equation}\label{alphaconv}
    u_i \to u_* \quad \text{in $C^{0,\alpha}_{loc}(\ball 1 \cup D_1)$ for every $0<\alpha<1$} 
\end{equation} 
Combining the gradient bound with the boundary condition of equation (\ref{eqn:conveqn}), we get either $u_*=0$ or $|u_*|=  1$ on $D_1$, but since $|u_i| \geq 1/2$ on $D_1$ so $u_* \neq 0$. Further we just saw that $u_*$ is continuous up to the boundary, so $u_*=1$ or $u_*=-1$. Without loss of generality we will assume that $u_*=1$ on $D_1$. Now it just remains to upgrade the convergence from $C^{0,\alpha}_{loc}$ to $C^{1,\alpha}_{loc}$ for points on $D_1$. If we show that 
$$
\norm{\frac{1}{\ve_i}(\ui-\ui^3)}_{C^{0,\alpha}(D_{1/4})} \leq C_{\alpha}
$$
then by (\ref{csmest}), $\norm{u_k}_{C^{1,\alpha}({\ball {1/8} \cup D_{1/8}})} \leq C_{\alpha}$ and we will have the desired convergence. To prove the above, due to (\ref{eqn:bdry}) and $u_*=1$ we have $u_i\geq 1/2$ for $i$ large enough, therefore it is enough to show that
\begin{equation} \label{eqn:convlemmaest}
\norm{\frac{1}{\ve_i}(1-\ui)}_{C^{0,\alpha}(D_{1/4})} \leq C_{\alpha}
\end{equation}
That is, we need to show that
\begin{equation} \label{goal}
     |1-\ui(x)| + \frac{|u_i(x+y)-u_i(x)|}{|y|^\alpha} \leq C_\alpha \ve_i \quad \text{for every } x,y \in D_{1/4}
\end{equation}
In the rest of the proof we establish this. First note that the estimates in (\ref{eqn:bdry}) combined with the boundary condition (\ref{eqn:conveqn}) gives us the first term in (\ref{goal}) 
\begin{equation}\label{eqn:hest}
    0 \leq 1-\ui(x) \leq \frac{8}{3}\ve_i \quad \text{for all } x \in D_1
\end{equation}

Now fix a $y \in \{D_{1/4} \backslash \{0\} \}$, then for $x\in \overline{\ball {1/2}}$ we write $u_{i,y}(x)=u_i(x+y)$ and $g_{i,y}(x) = u_{i,y}(x)-u_i(x)$. To get the desired second term in (\ref{goal}) we need to show that  $$|g_{i,y}| \leq \tilde{C}_\alpha  |y|^{\alpha}\ve_i \quad \text{on } D_{1/4} $$
For this, we will construct a function $v_i$ such that on 
$$
    |v_i|<C\ve_i \text{ and } v_i \pm g_{i,y}/|y|^\alpha \geq 0 \quad \text{on }D_{1/4}
$$ 
First note that $g_{i,y}$ solves,
\begin{equation} \label{geqn}
    \begin{cases}
        \Delta g_{i,y} = 0 \qquad \text{in } \ball {1/2} \\
        \frac{\ve_i}{u_i(u_i+1)} \pdv{g_{i,y}}{\nu}+ g_{i,y} = \ve_i  f_{i,y} \quad \text{on } D_{1/2}
    \end{cases}
\end{equation}
Here 
\begin{align*}
    f_{i,y} &= \bigg [ \frac{1}{u_{i,y} (1+u_{i,y})}  -\frac{1}{u_i(1+u_i)} \bigg] \pdv{u_{i,y}}{\nu} \\
    &= (u_i-u_{i,y})\bigg  [ \frac{1+\ui+u_{i,y}}{\ui u_{i,y}(1+\ui)(1+u_{i,y})} \bigg] \pdv{u_{i,y}}{\nu}
\end{align*}  
Using (\ref{eqn:bdry}), we estimate $f_{i,y}$ on $D_{1/2}$. 
\begin{align*}
    |f_{i,y}| &\leq |u_{i,y}-u_i| \cdot \frac{16}{3} \bigg|\pdv{u_{i,y}}{\nu} \bigg| \quad \because 1/2 \leq u_i \leq 1 \text{ on } D_1  \\
    &\leq |u_{i,y}-u_i|  \frac{32}{3} \qquad \because |\nabla u_i| \leq 2 \\
    &\leq C_\alpha |y|^{\alpha} \qquad \text{by } (\ref{alphaconv})
\end{align*}
Therefore,
\begin{equation}\label{eqn:gfest}
    \norm{f_{i,y}}_{L^{\infty}(D_{1/2})} < C_{\alpha}|y|^\alpha  
\end{equation}
Further, again by (\ref{eqn:bdry}), we have the following,
\begin{align} \label{supest}
    \norm{g_{i,y}}_{L^{\infty}(\overline{\ball {1/2}})} =C_y &\leq 2|y| \quad \because |\nabla u_i| \leq 2  \\
    &\leq C'_{\alpha}|y|^{\alpha} \text{ for some }C'_\alpha \because |y|<1/4 \label{cprimeeqn}
\end{align}
Now, we construct the function $\vi$ mentioned above. Consider the solution of the following mixed boundary value problem on $\ball {1/2}$
\begin{equation}\label{veqn}
    \begin{cases}
        \Delta \vi =0 \quad \text{in } \ball {1/2} \\
        \vi = 1  \quad \text{on } \partial^{+}\ball {1/2} \\
        \frac{4\ve_i}{3} \pdv{\vi}{\nu} + \vi = 0 \quad \text{in } D_{1/2} 
    \end{cases}   
\end{equation}
Clearly $\vi \leq 1$. However, $\vi \geq 0$ as well because if not then the point of negative minimum $x \in D_{1/2}$ but then Hopf boundary point lemma at $x$ gives $\ve_i \pdv{\vi}{\nu} + \vi <0$, which contradicts the boundary condition on $D_{1/2}$. So $|\vi|\leq 1$ and by lemma 6.26 in \cite{gt}, we have $$|\nabla \vi| \leq C \quad \text{in } \overline{\ball {1/4}}  $$
Here $C$ is a dimensional constant. Combining this with the boundary condition on $D_{1/2}$ for $\vi$ we get 
\begin{equation}\label{veboundary}
    |\vi| \leq C\ve_i \quad \text{on } D_{1/4}
\end{equation}
Now we have all the required estimates. Fix $\alpha \in (0,1)$, and recall that $y\in D_{1/4} \backslash \{0\}$ is fixed. Recall the definitions of $C_\alpha$, $C'_\alpha$ and $C_y$ from (\ref{eqn:gfest}) and (\ref{supest}). Now, consider the functions $$w_{i,\alpha,y}^{\pm} = C_y \vi   \pm g_{i,y} + \ve_i C_\alpha|y|^{\alpha}$$ 
We claim that $w_{i,\alpha,y}^{\pm}$ satisfy 
\begin{equation} \label{weqn}
    \begin{cases}
        \Delta w_{i,\alpha,y}^{\pm}=0\quad  \text{in } \ball {1/2} \\
        w_{i,\alpha,y}^{\pm} \geq \ve_i C_\alpha|y|^{\alpha} > 0 \quad \text{on } \partial^{+}\ball {1/2} \\
        \frac{\ve_i}{u_i(1+u_i)} \pdv{w_{i,\alpha,y}^{\pm}}{\nu} + w_{i,\alpha,y}^{\pm}  \geq 0 \quad \text{in } D_{1/2}
    \end{cases}
\end{equation}
First one is clear from definition as both $v_i$ and $g_{i,y}$ are harmonic and $\ve_i C_\alpha|y|^\alpha$ is constant as $y$ is fixed. For the second one we have 
\begin{align*}
    w_{i,\alpha,y}^{\pm} &= C_y v_i \pm g_{i,y} +\ve_i C_\alpha|y|^{\alpha} \\
    &\geq (C_y - |g_{i,y}|\,) +\ve_i C_\alpha|y|^{\alpha}  \quad \text{by (\ref{veqn}) on }\partial^+ \ball {1/2} \\
    &\geq \ve_i C_\alpha|y|^{\alpha} > 0 \quad \text{ by }(\ref{supest})
\end{align*}
For third  one we first note that on $D_{1/2}$ we have
\begin{align}
    \frac{\ve_i}{\ui(\ui+1)}\pdv{v_i}{\nu} +\vi
    &=\frac{4\ve_i}{3} \pdv{\vi}{\nu} + \bigg(\frac{1}{\ui(\ui+1)} - \frac{4}{3} \bigg)\ve_i \pdv{\vi}{\nu}  +\vi   \\
    &\geq \frac{4\ve_i}{3}  \pdv{v_i}{\nu} +\vi =0 \text{ by (\ref{veqn})} \label{vis0}
\end{align}
The inequality holds because the term in bracket is non-positive as $1/2\leq u_i \leq 1$ (\ref{eqn:bdry}), and because $\pdv{\vi}{\nu}\leq 0$ by (\ref{veqn}) as $\vi \geq0$ on $D_{1/2}.$ Hence, we have
\begin{align*}
    \frac{\ve_i}{\ui(1+\ui)} \pdv{w_{i,\alpha,y}^{\pm}}{\nu} + w_{i,\alpha,y}^{\pm} &\geq \pm \bigg(\frac{\ve_i}{\ui(\ui+1)} \pdv{g_{i,y}}{\nu}+ g_{i,y} \bigg)  + \ve_i C_\alpha|y|^{\alpha} \text{ by (\ref{vis0})}\\
    &=  \pm f_{i,y} + \ve_i C_\alpha|y|^{\alpha} \text{ by (\ref{geqn}) on } D_{1/2} \\ 
    &=\ve_i \big[C_\alpha|y|^{\alpha} \pm f_{i,y} \big] \quad  \\
    &\geq \ve_i \big[C_\alpha|y|^{\alpha}-C_\alpha |y|^\alpha \big] \quad \text{ by (\ref{eqn:gfest})} \\
    &\geq 0    
\end{align*}
    
Therefore $w_{i,\alpha,y}^{\pm}$ satisfies (\ref{weqn}). Observe that if $w_{i,\alpha,y}^{\pm} < 0$, then there is a point of negative minimum $x \in D_{1/2}$ as $w^{\pm}_{i,\alpha,y}\geq 0$ on $\partial^+ \ball {1/2}$. By Hopf boundary point lemma at $x$, this contradicts the boundary condition on $D_{1/2}$. So, $w_{i,\alpha,y}^{\pm} \geq 0$ in $\overline{\ball {1/2}}$ and therefore we have
\begin{align*}
    |g_{i,y}(x)| &\leq C_y \vi(x) + \ve_iC_{\alpha}|y|^{\alpha}  \quad \text{for } x\in \overline{\ball {1/2}} \\
    &\leq \ve_i|y|^{\alpha} \big[ C'_\alpha C + C_\alpha \big] \quad \text{for } x\in D_{1/4} \text{ by (\ref{cprimeeqn}) and (\ref{veboundary}) }
\end{align*}
Combining with estimates  (\ref{eqn:hest}),(\ref{supest}), and (\ref{veboundary}), we get
\begin{equation*}
    \frac{1}{\ve_i}\bigg|1-u_i(x)\bigg| + \frac{1}{\ve_i}\bigg|\frac{u_i(x+y)-u_i(x)}{|y|^\alpha} \bigg| \leq \frac{8}{3} + CC'_\alpha + C_\alpha \quad \text{for every } x,y \in D_{1/4}
\end{equation*}
This is the required estimate (\ref{eqn:convlemmaest}).
\end{proof}

\section{Small energy regime}
In this section we prove two results under a small energy assumption. The first one is a clearing out result on the boundary. The second is the epsilon regularity result. We emphasize that these are proved for $\ue$, with estimates that are uniform in $\ve$. This is crucial for their application in the study of $\ue$ as $\ve \to 0$.

\subsection{Clearing out on the boundary} This is an important consequence of the monotonicity formula. It captures the intuition that if the energy is small enough then $\ue$ stays close to the potential wells uniformly with respect to $\ve$.
\begin{lemm} \label{lemmaclearing}
Let $\ue$ be a solution of (\ref{eqn:main}) for $R=1$ such that $|\ue| \leq 1$. There is a constant $\eta$ independent of $\ve$ such that such that $E_\ve(\ue,\ball 1)\leq \eta$ implies $|\ue| \geq \frac{1}{2}$ on $D_{1/2}$.
\end{lemm}
\begin{proof}

We first consider solution $u$ for $\ve =1$ and prove the claim by contradiction. Then we will prove it for any $\ve<1$ by rescaling. As $|u| \leq 1$, we have by (\ref{csmest}) that $\norm{u}_{C^{2,\alpha}(\ball {1/2})} \leq C_{\alpha}$. Let $\eta_1=2^{n-1}\eta$ such that $E_\ve(\ue,\ball 1)\leq \eta_1$. If the result is not true then we can find a sequence of solutions $\ui = u_{\ve_i}$ of decreasing energy and points $x_i \in D_{1/2}$ such that $|\ui(x_i)|\leq 1/2$ and $E_1(u_i,\ball 1) \to 0$. Due to the uniform estimate we have uniform convergence of $u_i$ in $\overline{\ball {1/2}}$. As $E_1(u_i,\ball 1) \to 0$, $u_i \to 1$ on $D_{1/2}$ which contradicts $|\ui(x_i)|\leq 1/2$. This proves the case $\ve=1$. 

Now for any $\epsilon<1$, for any solution $\ue$ and any point $x_0 \in D_{1/2}$,  consider the map $x \to x_0 + \ve x$ sending $\ball 1 \to \ball \ve (x_0)$. Then $v_{\Tilde{\ve}}(x)=\ue(x_0+\ve x)$ satisfies (\ref{eqn:main}) for $\ve=1$ and $R=1$. Further note that $E_1(v_{\Tilde{\ve}}, \ball 1)= I_{\ve}(\ve,x_0) \leq 2^{n-1} \eta=\eta_1$, the inequality is by (\ref{eqn:mflemma}). So we can apply the $\ve=1$ result which gives $|v_{\Tilde{\ve}}| \geq 1/2$ on $D_{1/2}$, but this is same as $|\ue| \geq 1/2$ on $D_{\frac{\ve}{2}}(x_0)$ for any $x_0 \in D_{1/2}$.
\end{proof}

\subsection{Epsilon regularity} We now prove the epsilon regularity result. It will become clear soon that it should be thought of as stating that if the energy bound is small enough, then the gradient of $\ue$ is bounded uniformly independent of $\ve$ \textit{up to the boundary.}

\begin{theo}  There exist constants $\eta_0$ and $C_0$ independent of $\ve$ such that for $\ve <R$, and $u_\ve \in C^2(\ball R \cup D_R)$ satisfying $|u_\ve| \leq 1$ solving 

    \begin{equation} \label{eqn:Rscale}
        \begin{cases}
            \Delta \ue = 0 \qquad &\text{in } B_R^+ \\
            \pdv{\ue}{\nu} = -\frac{1}{\ve}W'(u_\ve) \qquad  &\text{on  } D_R
        \end{cases}
    \end{equation}
If we have $I_{\ve}(R,0) \leq \eta_0$, then 
    \begin{equation} \label{epsreg}
        \sup_{\ball {\frac{R}{4}}} |\nabla \ue|^2 + \sup_{D_{\frac{R}{4}}} \frac{W(\ue)}{\ve ^2} \leq \frac{C_0}{R^2}\eta_0
    \end{equation}
\end{theo}
\begin{rem}
We use an idea due to Schoen \cite{s} in harmonic maps setting, also used in similar geometric problems by several authors \cite{cs,cw,ms,t,w}. To obtain an estimate independent of $\ve$ we need a scale $r_{\ve}$ for which the gradients are uniformly bounded. Then the problem reduces to having a uniform bound on $r_{\ve}$. In \cite{s} this is done by using the mean value property to get a contradiction to the smallness of energy. However, due to the boundary we may only use the mean value property for $|\nabla \ue|^2$ only for points sufficiently away from the boundary. In the other case, we will rescale the problem to $r_{\ve}$-scale and then use the convergence lemma \ref{lem:conv} to get a contradiction to smallness of energy
\end{rem}

\begin{proof}
It is sufficient to prove the result for $R=1$. First observe that we have 
$$\frac{W(u)}{\ve^2}=\frac{1}{4u^2}\frac{W'(u)^2}{\ve^2} = \frac{1}{4u^2}\bigg|\pdv{\ue}{\nu}\bigg|^2 $$
We may assume that $I_\ve \leq \eta$, then by clearing out result (\ref{lemmaclearing}), $1/2 \leq|u_\ve|\leq 1$ on $D_{1/2}$ and therefore
\begin{equation*}
    \frac{W(u)}{\ve^2}= \frac{1}{4u^2}\bigg|\pdv{\ue}{\nu}\bigg|^2  \leq \bigg|\pdv{\ue}{\nu}\bigg|^2
\end{equation*}

Therefore to establish (\ref{epsreg}), it is enough estimate the gradient $|\nabla \ue|$ upto the boundary, i.e. we need to show
\begin{equation} \label{gradest}
    \sup_{\overline{\ball {\frac{1}{4}}}} |\nabla \ue| \leq C\sqrt{\eta_0} 
\end{equation}

Consider the distance weighted gradient on $\overline{\ball {\frac{1}{2}}}$,  $F(s)=(\frac{1}{2}-s)|\nabla \ue(x)|$. It attains its maximum for some $s_{\ve} \in [0,1/2]$ and we have 
$$\max _{s} F(s) = \max_{s} \bigg(\frac{1}{2}-s \bigg) \sup_{\overline{\ball s}}|\nabla \ue| =\bigg(\frac{1}{2}-s_{\ve}\bigg) \sup_{\overline{\ball {s_{\ve}}}} |\nabla \ue| $$

Let $x_\ve$ be such that $\sup_{\overline{\ball {s_{\ve}}}} |\nabla \ue|=|\nabla \ue(x_\ve)| =e_\ve$. Note that due to the definition of $F(s)$, $|x_{\ve}|=s_{\ve}$, so $\dist{(x_\ve, \partial^+ \ball \frac{1}{2})}= \frac{1}{2}-s_\ve$. We write $$ \frac{1}{2}-s_\ve = 2\rho_\ve \quad \text{and} \quad r_\ve = \rho_\ve e_\ve$$

Then we have $\max_{s} F(s) =2\rho_\ve  e_\ve = 2r_\ve $. Taking $s=\frac{1}{4}$ gives
\begin{equation} \label{eqn:grad1}
    \sup_{\overline{\ball {\frac{1}{4}}}} |\nabla \ue| \leq 8 r_\ve
\end{equation}
Therefore the gradient estimate would follow from a uniform bound on $r_\ve$. We collect another consequence of definition of $F$ that will be used later.
\begin{equation} \label{eqn:gs}
    |\nabla \ue|(x) \leq \frac{(\frac{1}{2}-s_\ve)}{(\frac{1}{2}+s_\ve)}2 e_\ve < 2e_\ve \quad \text{for all } x \in B_{\rho_\varepsilon}(x_\ve) \cap \overline{\ball {\frac{1}{2}}} 
\end{equation}

Let $\overline{x_\ve}$ be projection of $x_\ve$ on $D_1$. Then denote by $z_\ve$ the height of $x_\ve$, i.e $z_\ve = x_\ve - \overline{x}_{\ve}$. Depending on how the height of $x_\ve$ compares to it is distance from the spherical boundary we get the above described two cases.

      \textit{Case 1:}   $ \frac{z_\ve}{2\rho_\ve} > \frac{1}{4}$ i.e., away from  $D_1$\\
           The ball $B_{\frac{z_\ve}{2}}(x_\ve) \subset \ball {2z_\ve}(\overline{x}_\ve) \subset \ball 1 $, then by mean value property for $|\nabla u_\ve|^2$ we have 
        \begin{equation} \label{eqn:case1}
            e_\ve ^2 \leq \frac{1}{|B_{\frac{z_\ve}{2}}(x_\ve)|}\int_{B_{\frac{z_\ve}{2}}(x_\ve)} |\nabla u_\ve|^2 \,dx 
                    \leq \frac{1}{4z_{\ve}^2}\cdot C\eta_0  
        \end{equation}
        The second inequality follows from (\ref{eqn:mflemma}). Combining this with (\ref{eqn:grad1}) gives the desired estimate
        $$ \sup_{\overline{\ball {\frac{1}{4}}}} |\nabla \ue| \leq 8 r_\ve < 8\cdot  2z_\ve e_\ve \leq C\sqrt{\eta_0} $$   

\textit{Case 2: }$ \frac{z_\ve}{2\rho_\ve} \leq \frac{1}{4}$ i.e., close to $D_1$\\
            Given (\ref{eqn:grad1}), if $r_\ve \leq 1$, then we are done. So we assume otherwise, i.e $r_\ve > 1$ and arrive at a contradiction. For this, we will consider the problem at the $r_\ve$-scale. Consider $\ue$ solving (\ref{eqn:Rscale}) in the ball $\ball {\rho_{\ve}}(\overline{x}_\ve) \subset \ball {\frac{1}{2}} $. Then for $x \in \ball {r_\ve} \cup D_{r_\ve}$, with $\Tilde{\ve} = \ve e_\ve$, take
            $v_{\Tilde{\ve}}(x) = \ue (\overline{x}_\ve + x/e_\ve) $. With this rescaling $\ball {\rho_\ve}(\overline{x}_\ve)$ goes to $\ball {r_\ve}$. As $r_\ve > 1$, all $v_{\Tilde{\ve}}$ solve
            \begin{equation} \label{eqn:rscale}
        \begin{cases}
            \Delta v_{\Tilde{\ve}} = 0 \qquad &\text{in } \ball {1} \\
            \pdv{v_{\Tilde{\ve}}}{\nu} = -\frac{1}{\Tilde{\ve}}W'(v_{\Tilde{\ve}}) \qquad  &\text{on  } D_{1}
        \end{cases}
    \end{equation}
           Further due to our assumptions, and (\ref{eqn:gs}) we have
        \begin{equation} \label{eqn:convestimates}
            |\nabla v_{\Tilde{\ve}} (y_\ve)|=1 ,|\nabla v_{\Tilde{\ve}}| \leq 2,|v_{\Tilde{\ve}}| \leq 1 \quad \text{in } B_1^+ \cup D_1 \text{and } |v_{\Tilde{\ve}}|\geq 1/2 \quad \text{on } D_1 
        \end{equation}
        further using (\ref{eqn:mflemma}) we also get
        \begin{equation}\label{eqn:energyconv}
            E_{\Tilde{\ve}}(v_{\Tilde{\ve}}, B_1)<E_{\Tilde{\ve}}(v_{\Tilde{\ve}}, B_{r_\ve})=I_\ve(\rho_\ve,\overline{x}_\ve) \leq 2^{n-1} \eta_0
        \end{equation}
            Here $y_\ve = e_\ve(x_\ve-\overline{x}_\ve)$, so $|y_\ve|= z_\ve e_\ve$. Note that $z_\ve \leq \frac{1}{2}\rho_\ve$ is equivalent to $z_\ve e_\ve \leq r_\ve/2$ but we need $y_\ve \in \ball {1/2}$, indeed this is the case i.e.   $z_\ve e_\ve \leq 1/2$ as if we have $z_\ve > \frac{1}{2e_\ve}$ then exactly like (\ref{eqn:case1}) we have
            $$ e_\ve^2 \leq \frac{1}{4z_\ve^2}\cdot C\eta_0 \leq e_\ve^2\cdot C\eta_0$$ this gives $1 \leq C\eta_0 $
            which is false for $\eta_0$ small enough. Therefore, $y_\ve \in \ball {1/2}$. To simplify notation we write, $v_i=v_{\Tilde{\ve_i}}$ and $y_i=y_{\ve_{i}}$. 
            
            We will show that as $i \to \infty$, we have $v_i \to v_*$ $C^{1,\alpha}_{loc}(\ball 1 \cup D_1)$, then it will give us
            
            \begin{equation}\label{assump}
                \begin{cases}
                 \nabla v_i(y_i) \to \nabla v_*(y_*) \text{ thus } |\nabla v_*(y_*)|=1 \text{ as } |\nabla v_i(y_i)|=1  \\
               \text{There is a }\sigma<\frac{1}{10} \text{ such that }|\nabla v_*|  \geq 1/2 \text{ on }B_{\frac{1}{10}}(y_*)\cap \ball 1 
            \end{cases}
            \end{equation}
    This gives us the following. The last inequality is due to the estimate (\ref{eqn:energyconv}).
        \begin{equation}
            \frac{\sigma^{n+1}|\ball 1|}{2} \leq \int_{\ball 1}|\nabla v_*|^2\,dx \leq \liminf_{i\to \infty} \int_{\ball 1} |\nabla v_i|^2 \,dx \leq 2^{n-1} \eta_0
        \end{equation}    

This leads to a contradiction for small enough $\eta_0$. Therefore $r_\ve <1$ as desired.

We now just need to show that $v_i \to v_*$ in $C^{1,\alpha}_{loc}(\ball 1 \cup D_1)$.  Recall that $\Tilde{\ve_i}=\ve_i e_{\ve_i}$. We have two cases: As $\ve_i \to 0$ we also have $\Tilde{\ve}_i \to 0$ Then because of the estimates (\ref{eqn:convestimates}), we can apply lemma (\ref{lem:conv}) to $\{v_i\}$.This gives uniform $C^{1,\alpha}_{loc}$ convergence $v_i$ to $v_*$ up to the boundary as required. However, if as $\ve_i \to 0$, $\Tilde{\ve}_i \nrightarrow 0$, that is $\Tilde{\ve}_i= e_{\ve_i} \ve_i \geq \beta >0$. Then by (\ref{csmest}) we have uniform $C^{2,\gamma}$ estimate up to the boundary, $\norm{v_i}_{C^{2,\gamma}_{loc}(\overline{\ball {3/4}}) } \leq C_{\gamma}$. By Holder interpolation, this gives uniform $C^{1,\alpha}_{loc}$ convergence $v_i$ to $v_*$ up to the boundary in this case as well. This finishes the proof.
\end{proof}

With the monotonicity formula, convergence lemma and the epsilon regularity result we can now study the behavior of $\ue$ and the associated energy as $\ve \to 0$. 
\section{The energy concentration set and its properties}
In this section we introduce the concentration set of the energy and prove several results about it using the tools developed in the last two sections. Throughout this section, we will have the following assumptions unless otherwise stated. $\{u_i\}_{i\in \NN} \subset H^1(U)$ are the critical points of $E_{\ve_i}$, satisfying

 \begin{equation} \label{sec4eqn}
        \begin{cases}
            \Delta u_i = 0 \qquad &\text{in } U \\
            \pdv{u_i}{\nu} = -\frac{1}{\ve_i}W'(\ue) \qquad  &\text{on  } \partial^0 U \\
            |u_i|\leq 1$, $E_{\ve_i}(u_i) \leq E_0
        \end{cases}
\end{equation}

\subsection{Limiting energy measure and its density}Consider the energy measures on $U \cup \partial^0U$,
$$ \mu_i = \frac{1}{2}|\nabla u_i|^2 \,dx + \frac{1}{\ve_i}W(u_i) \,\mh $$
Because $\mu_i(U \cup \partial^0U) \leq E_0$, after passing to a subsequence there exists a Radon measure $\mu$ on $U\cup \partial^0U$, such that $ \mu_i \rightharpoonup \mu $. For any $x \in \partial^0U$, and $0<s<r<\dist(x,\partial^+U)$, using the monotonicity formula (\ref{eqn:mf1}), with $i \uparrow \infty$ gives 
$$ \frac{1}{s^{n-1}} \mu(\ball s (x)) \leq \frac{1}{r^{n-1}} \mu(\ball r (x)) < \frac{E_0}{r^{n-1}} $$
Hence, the $(n-1)$ density of the measure $\mu$ is well-defined and finite 
\begin{equation}
    \Theta^{n-1}(\mu,x) = \lim_{r \downarrow 0} \frac{\mu(\ball r (x))}{\omega_{n-1}r^{n-1}}
\end{equation}
Here, $\omega_{n-1}$ is the volume of $(n-1)$ dimensional unit ball. We will write $\omega_{n-1}$ is $1$ to simplify notation. We first observe that the density is zero for points in the interior. 

\begin{prop} \label{lem0}
 If $x \in U$ then $\Theta^{n-1}(\mu,x)=0$.
\end{prop} 
\begin{proof}
    First note that the uniform energy bound and $|u_i| \leq 1$, gives us a weak limit $u_* \in H^1(U)$ for $\{u_i\}$. After passing to a subsequence $|\nabla u_i|^2 \to |\nabla u_*|^2$ pointwise almost everywhere so by Fatou's lemma we have,
    $$\int_U |\nabla u_*|^2 \,dx \leq \lim_{i \to \infty} \int_U |\nabla u_i|^2 \,dx$$
    We now show that this is actually an inequality for proper subsets of $U$. First we see that $u_*$ is not just a weak $H^1$ and pointwise limit, but in fact, $u_*$ is harmonic in $U$ and $u_i \to u_*$ in $C^{\infty}_{loc}(U)$ as $u_i$ are harmonic in $U$ and satisfy $|u_i| \leq 1$. Therefore, for a $x \in U$, let $r<\dist(x,\RR^n)$. Then on the ball $B_r(x)$, the $C^1_{loc}$ convergence gives equality in Fatou's lemma
    \begin{equation} \label{intconvu}
       \int_{B_r(x)} |\nabla u_*|^2 \,dx = \lim_{i \to \infty} \int_{B_r(x)} |\nabla u_i|^2 \,dx 
    \end{equation}
    and since $\mu_i(B_{r}(x)) = \int_{B_r(x)} |\nabla u_i|^2 \,dx$ we have 
    \begin{equation} \label{intmeasure}
        \mu(B_r(x)) = \int _{B_r(x)}|\nabla u_*|^2 \,dx
    \end{equation} 
    Therefore, we have
    \begin{align*}
        \Theta^{n-1}(\mu,x)&= \lim_{r \to 0}\frac{1}{r^{n-1}} \int_{B_r(x)} |\nabla u_*|^2 \,dx \\
        &\leq C_{n+1} \lim_{r\to0} r^{2} \quad \text{since }u_* \text{ is harmonic and }|u_*|\leq 1 \\
        &=0
    \end{align*}
\end{proof}

Now we consider $\Theta^{n-1}(\mu,x)$ for points on the boundary, $x \in \partial^0U$. As density is upper semicontinous we have for each $i$ and small enough $r>0$ a closed subset 
\begin{equation} \label{rel:concset0}
    \Sigma_{i,r} = \{x \in \partial^0U : r^{1-n} \mu_i(\ball r (x)) \geq \eta_0\}
\end{equation}

Here $\eta_0$ is the epsilon regularity threshold from Theorem \ref{epsreg}. From monotonicity we have $\Sigma_{i,s}\subset \Sigma_{i,r}$ for $s<r$. After a diagonal argument we can pass to a subsequence such that(without relabeling) $\Sigma_{i,2^{-k}}$ converges in a hausdorff distance sense to $\Sigma_{2^{-k}}$. Note that for $k<l$, $\Sigma_{2^{-l}} \subset \Sigma_{2^{-k}}$, and we set $\Sigma = \cap_{k} \Sigma_{2^{-k}}$. By construction $\Sigma$ is the concentration set of the energy, i.e set of all points where the density is not zero, hence above the epsilon regularity threshold $\eta_0$, namely 
\begin{equation} \label{rel:concset}
    \Sigma = \{x \in \partial^0 U : \Theta^{n-1}(\mu,x) \geq \eta_0 \}
\end{equation}

We now show that the concentration set $\Sigma$ is a $(n-1)$ dimensional object in $\RR^n \cong \partial \RR^{n+1}$
\begin{lemm}
  For $\Sigma$ defined in (\ref{rel:concset})  $\mathcal{H}^{n-1}(\Sigma) < \infty$. 
\end{lemm}
\begin{proof}
    Let $K$ be any compact subset of $U \cup \partial^0 U$. Let $d_K= \dist(K, \partial^+U)$. Then, for any $0<r< d_K$, we can take a finite subcover of $\Sigma\cap K$, $\{\ball {r_k}(x_k)\}$ such that $\{\ball {r_k/2}(x_k) \}$ are disjoint, $r_k < r$, and $x_k \in \Sigma$.  Then by (\ref{rel:concset}) and monotonicity formula we have for large enough $i$
    \begin{equation}
        \eta_0 \leq \frac{1}{(r_k/2)^{n-1}} \mu_{i}(\ball {r_k}(x_k)) \leq  \frac{C_n}{r_k^{n-1}} E_0   
    \end{equation}
    Summing over $k$ gives us          
   \begin{equation*}
        \sum _{k} r_k^{n-1} \leq \frac{1}{\eta_0} C_n E_0 
    \end{equation*}
Therefore, for all compact sets $K$ we have $\mathcal{H}^{n-1}(\Sigma\cap K) \leq \frac{1}{\eta_0} C_n E_0 $, hence
$$\mathcal{H}^{n-1}(\Sigma) \leq \frac{C_n}{\eta_0}E_0 < \infty$$
\end{proof}

\subsection{Structure of limiting measure}
We now prove that the potential vanishes in the limit. This will allow us to clearly describe the relationship between limit measure $\mu$ and the limiting function $u_*$.
\begin{prop} \label{vanishing}
Let $\{u_i\}_{i \in \NN}$ satisfy the conditions in (\ref{sec4eqn}), then as Radon measures on $\partial^0 U$,
$$\frac{W(\ui)}{\ve_i}\,d\mathcal{H}^n \rightharpoonup 0 $$
\end{prop}
\begin{proof}
    Recall that the set $\Sigma$ is closed. So for any $x\notin \Sigma$ and $r < \dist(x,\Sigma \cup \partial^+ U)$ we have by $(\ref{rel:concset0})$ and $(\ref{rel:concset})$, for all $i$ large enough , $r^{1-n} \mu_i(\ball r (x) ) \leq \eta_0$. Hence, epsilon regularity (\ref{epsreg}) gives 
$$\frac{1}{\ve_i} W(u_i)  \leq \frac{C}{r^2} \ve_i \quad \text{in } D_{r/4}(x) $$
and therefore
$$\int_{D_{r/4}(x)}\frac{W(\ui)}{\ve_i}\,d\mathcal{H}^n \leq Cr^{n-2}\ve_i \to 0  \quad \text{as }\ve_i \downarrow 0 $$
In view of this we consider a countable covering of $\partial^0U \backslash \Sigma$: for every $k\in \NN$, there are only finitely many points $\{x_{j,k}\}$, such that $D_{j,k}=D_{2^{-k}}(x_{j,k})$ is disjoint with all balls until the $(k-1)$ step, and $4\cdot 2^{-k} < \dist(x,\Sigma \cup \partial^+ U)$. Take $D_k = \cup_{j} D_{j,k}$(note that this is a finite union). By construction we have
$$ \coprod_{k \in \NN} D_k = \partial^0U \backslash \Sigma $$
and also that
$$ \lim_{i \to \infty} \int_{D_k} \frac{1}{\ve_i} W(u_i)  \mh = 0 \quad \text{for all } D_k   $$
Note that $H^{n-1}(\Sigma)< \infty$  therefore $H^n(\Sigma)=0$ and we have
\begin{align*}
    \lim_{i \to \infty} \int_{\partial^0 U} \frac{1}{\ve_i} W(u_i) \mh  &= \lim_{i \to \infty} \int_{\partial^0 U \backslash \Sigma} \frac{1}{\ve_i} W(u_i) \mh  \\
     &= \lim_{i \to \infty} \sum_{k} \int_{D_{k}} \frac{1}{\ve_i} W(u_i)  \mh \\
     &= 0 \quad \text{by Dominated Convergence Theorem}
\end{align*}

\end{proof}

We can now give a complete description of the limiting measure $\mu$ and it's relationship with $\Sigma$ and $u_*$. As Radon measures on $U \cup \partial^0U$, we have $\mu_i \rightharpoonup \mu $. However as the potential vanishes in the limit, we get
$$\frac{1}{2}|\nabla u_i|^2 \,dx \rightharpoonup \mu$$
Now, as we saw in the proof of lemma \ref{lem0}, by Fatou's lemma 
\begin{equation} \label{fatou}
\int_U \frac{1}{2}|\nabla u_*|^2 \,dx \leq \lim_{i \to \infty} \int_U \frac{1}{2}|\nabla u_i|^2 \,dx     
\end{equation}

For sets in the interior the equality holds, but in general it does not hold. We first prove a lemma describing the limit function $u_*$ on the boundary. 

\begin{prop} \label{onu*}
    As $\ve_i \to 0$, we have $u_i \to u_*$ in $C^{1,\alpha}_{loc}(U\cup \partial^0U \backslash \Sigma)$, for $\alpha \in (0,1)$. Further, $u_*=1$ or $-1$ on each connected component of $\partial^0U \backslash \Sigma$.
\end{prop}
\begin{proof}
    The proof in the interior was subsumed in proof for proposition \ref{lem0}. Therefore let $x \in \partial^0U \backslash \Sigma$. Let $r = \frac{1}{2} \min(\dist(x,\Sigma),\dist(x,\partial^+U))$. Then on $\ball r (x)$, for $i$ large enough all $\{\ui\}$ satisfy $I_{\ve_i}(r,x) < \eta_0$. Therefore by epsilon regularity (\ref{epsreg}) and clearing out (\ref{lemmaclearing}) applied in $\ball r (x)$, the functions $\{\ui\}$ satisfy all hypothesis of the convergence lemma \ref{lem:conv} in $\ball {r/4}(x) $, and we obtain $\ui \to u_*$ in $C^{1,\alpha}_{loc}(U\cup \partial^0U \backslash \Sigma)$ and $u_*=\pm 1$.  
\end{proof}

The following corollary says that the equality in (\ref{fatou}) holds for a subset $V$ of
$U$ if it does not touch $\Sigma$. 

\begin{coro} \label{propfatou}
    Let $V$ be an admissible open set in $U$ such that $\overline{V} \cap \Sigma = \phi$, then 
    $$\int_V \frac{1}{2}|\nabla u_*|^2 \,dx = \lim_{i \to \infty} \int_V \frac{1}{2}|\nabla u_i|^2 \,dx  $$
\end{coro}
\begin{proof}
   We first assume that $V$ is a ball centered at $x$ and split the proof into two cases.
   
   Case 1: Let $V = B_r(x) \subset U$. In this case the conclusion was already established in (\ref{intconvu}) while proving lemma \ref{lem0}.$$
    \int_{B_r(x)} |\nabla u_*|^2 \,dx = \lim_{i \to \infty} \int_{B_r(x)} |\nabla u_i|^2 \,dx $$
   
   Case 2: When $V = \ball r (x)$. Due to the assumption $\overline{V} \cap \Sigma = \phi $, $x \in   \partial^0U \backslash \Sigma $ and as $\Sigma$ is closed, there is a half ball $\ball r (x)$ disjoint from $\Sigma$, such that we have uniform $C^{1,\alpha}(\overline{\ball {r/4}(x))}$ convergence $u_i$ to $u_*$ by (\ref{onu*}) hence $$\int_{\ball {r/4}(x)} |\nabla u_*|^2\,dx = \lim_{i \to \infty} \int_{\ball {r/4}(x)} |\nabla u_i|^2\,dx $$
   as desired, in this case as well. 
   Together, these two cases imply the result for every admissible open set $V$ by a covering argument similar to the one used in proof of Theorem \ref{vanishing}
\end{proof}

The above two results show us that the obstruction to equality in (\ref{fatou}) is the set $\Sigma$. That is the lack of compactness in $H^1(U)$, i.e inequality in (\ref{fatou}) is due to the fact that when the energy of $\{u_i\}$ is high, then the tendency of $|u_i|$ to converge to $1$, as $\ve_i \downarrow 0$, leads to loss of energy in the singular limit. This energy loss however is captured in the energy concentration set $\Sigma$. Indeed we may rewrite (\ref{fatou}) as,
\begin{equation} \label{measdecom}
    \frac{1}{2}|\nabla u_i|^2 \,dx \rightharpoonup \mu= \frac{1}{2}|\nabla u_i|^2 \,dx + \mu_{\Sigma}
\end{equation} 
where $\mu_{\Sigma}$ is the defect measure that arises if there is a failure of strong convergence in $H^1(U)$ as explained above. The following result gives a complete description of $\mu_{\Sigma}$.

\begin{theo}\label{thm:sing}
    The measure $\mu_{\Sigma}$ has the following properties. 
    \begin{enumerate}
        \item It is supported in the energy concentration set $\Sigma$.
        \item $\Theta^{n-1}(\mu_{\Sigma},x)=\Theta^{n-1}(\mu, x)$ for $\mathcal{H}^{n-1}$ a.e. on $\Sigma$. 
        \item Further writing $\theta(x) = \Theta^{n-1}(\mu_{\Sigma},x)$, we have  $$\mu_{\Sigma}=\theta \, \mathcal{H}^{n-1}\measurer \Sigma$$ with $\eta_0 \leq \theta(x) \leq C < \infty$, for $\mathcal{H}^{n-1}$ a.e. $x \in \Sigma$.
    \end{enumerate}

\end{theo}
\begin{proof}
Let $x \in U \cup \partial^0U \backslash \Sigma$. Then there is a ball(or half ball for boundary point) $B_x$ containing $x$, and disjoint from $\Sigma$, by corollary \ref{propfatou}, $$\mu(B_x)= \int_{B_x} |\nabla u_*|^2 \,dx$$
and therefore by (\ref{measdecom}), $\mu_{\Sigma}(B_x)=0$, for all such $x$, therefore $spt(\mu_{\Sigma}) \subset \Sigma$. This proves part (1). 

Next, as $u_* \in H^1(U)$ and note that $U \subset \RR^{n+1}$ therefore by equation (3.3.38) in \cite{z}, the $(n-1)$-density of $u_*$ is $\mathcal{H}^{n-1}$ a.e. $0$ i.e we have $$ \lim _{r \to 0}\frac{1}{r^{n-2}}\int_{B_r(x)} |\nabla u_*|^2\,dx = 0 \quad \mathcal{H}^{n-1} a.e. $$
In particular this implies part (2) i.e 
$$\Theta^{n-1}(\mu_{\Sigma},x)=\Theta^{n-1}(\mu, x)  \quad \mathcal{H}^{n-1} a.e. \text{ } x \in \Sigma $$
For part (3), recall that by definition $\eta_0 \leq \Theta^{n-1}(\mu, x)$ for all $x$ on $\Sigma$. The upper bound follows from monotoncity formula. Combined with part (2), this gives 
$$\eta_0 \leq \Theta_{\Sigma}^{n-1}(\mu, x) < C \quad \mathcal{H}^{n-1} a.e. \text{ } x\in \Sigma  $$
Therefore is absolute continous with respect to $\mathcal{H}^{n-1}\measurer \Sigma$ and by Radon-Nikodyn Theorem we get part (3), i.e  $\mu_{\Sigma} = \theta \, \mathcal{H}^{n-1}\measurer \Sigma$.
\end{proof}

\section{The limiting varifold}

\subsection{Rectifiable varifolds } We define rectifiable varifolds and  some notions related to them that we require in this paper. For a thorough treatment of rectifiable varifolds we refer the reader to chapter 4 in \cite{si}.
\begin{defi}
A set $\Sigma \subset \RR^{n+1}$ is $k$-rectifiable if and only if $\Sigma \subset \cup_{i=0}^{\infty} \Sigma_i$, where $\mathcal{H}^k(\Sigma_0)=0$ and for $i\geq 1$, each $\Sigma_i$ is an embedded $k$-dimensional  $C^1$ submanifold in $\RR^{n+1}$ \end{defi}

\begin{rem}
    An important property of $k$-rectifiable sets that we will be using is the existence of (approximate) tangent space $T_x \Sigma$ for $\mathcal{H}^k$ a.e point $x \in \Sigma$.
\end{rem}

We denote by $Gr_{k,n+1}$ the $k$- grassmann manifold that contains unoriented $k$ planes in $\RR^{n+1}$, and identify each subspace will the (symmetric) matrix of orthogonal projection on it. Then a $k$- varifold in $U \subset \RR^{n+1}$ is a Radon measure on $U \times Gr_{k,n+1}$. For our purpose we are interested in the smaller class of rectifiable varifolds, namely

\begin{defi}
Let $\Sigma$ be a $k$-rectifiable set in an open set $U \subset \RR^{n+1}$ and a positive Borel function $\theta: \Sigma \to \RR$. Then $V=V(\Sigma, \theta)$ is a $k$-rectifiable varifold, with multiplicity function $\theta$, given by 
\begin{equation*}
    \langle V(\Sigma,\theta),f \rangle = \int_{\Sigma} \theta(x) f(T_x \Sigma) \,d\mathcal{H}^k_x \quad \text{for any } f\in C^0(Gr_{k,n+1})
\end{equation*}
Further, if the multiplicity function $\theta \in \NN \backslash \{0\}$,   $\mathcal{H}^k$ a.e. on $\Sigma$, then $V$ is called \textit{integral} or \textit{integer rectifiable}. 
\end{defi} 
\begin{defi}
For any $k$-rectifiable varifold $V(\Sigma,\theta)$, the associated mass measure $\mu_V$ is a Radon measure given by,
$$ 
\mu_V(A) = \int_{A \cap \Sigma} \theta \,d\mathcal{H}^k \quad \text{ for any }\mathcal{H}^k \text{ measurable set } A 
$$
and the mass of $V$ is 
$$
\textbf{M}(V) = \mu_V(\RR^{n+1})= \int_{\Sigma} \theta \,d\mathcal{H}^k
$$
Note that this is just the $k$-area functional. We denote by $\Phi_{\#}V$ the pushforward of $V$ by a diffeomorphism $\Phi$.
\end{defi}
Given any vector field $X \in C_c^1(U, \RR^{n+1})$, denote by $\{\Phi_t\}$ the one-parameter family of diffeomorphisms generated by it. We can now define the first variation of $V$.
\begin{defi}
    Let $V(\Sigma,\theta)$ be a rectifiable $k$- varifold and $X \in C_c^1(U, \RR^{n+1})$, then the first variation of $V$ is 
    \begin{equation}\label{fvari}
      \langle \delta V , X \rangle = \frac{d}{dt}\bigg|_{t=0}\textbf{M}((\Phi_{t})_{\#}V) = \int_{\Sigma} \theta \D_{\Sigma} X \,d\mathcal{H}^k  
    \end{equation}
     
    The varifold is said to be \textit{stationary} if $\langle \delta V, X \rangle =0$ for all $X$.
\end{defi}
\begin{rem}
     Note that the stationarity condition just states that $V(\Sigma,\theta)$ is a critical point of the $k$-area functional under pushforwards by diffeomorphisms. Therefore, it is reasonable to think of stationary rectifiable varifolds as a weak notion of minimal surfaces. 
\end{rem}

\subsection{Generalized Varifolds}
In several geometric problems involving energy concentration, the stress energy tensors associated to the problem are not varifolds and therefore, it is necessary to consider a larger space than the space of varifolds to deal with such problems.  For this purpose, Ambrosio and Soner introduced the notion of \textit{generalized varifold} in \cite{as} in their study of parabolic ginzburg-landau equations. Precisely, instead of looking at Radon measures in $U \times Gr_{k,n+1}$, they take Radon measures in $U \times A_{k,n+1}$, where $A_{k,n+1}$ compared to $Gr_{k,n+1}$ is a slightly larger subset of symmetric matrices, namely 
$$A_{k,n+1}= \{A \in Sym_{n+1} \, | \, tr(A)=k, -(n+1) I_{n+1} \leq A \leq I_{n+1}\}$$

Generalized varifolds enjoy similar compactness properties as varifolds and therefore one obtains that a limiting generalized varifold on energy concentration. Since introduction, they have been used in several works like \cite{lw,ms} and recently in the work of Pigati and Stern \cite{ps} on minmax construction of codimension-$2$ integer rectifiable stationary varifolds.

\begin{defi}
  A generalized $k$-varifold in $U \subset \RR^{n+1}$ is a nonnegative Radon measure $V$ in $U \times A_{k,n+1}$.
\end{defi}

The notions of mass measure, first variation, stationarity and density extend in a straightforward manner from varifolds  to generalized varifolds. We refer to section 3 in \cite{as} for details. 

\subsection{The associated generalized varifolds}
We now see that the stress energy tensors associated to any $\ue$, satisfying (\ref{sec4eqn}) are $(n-1)$-generalized varifolds. Note that we will just write $A_{n-1}$ to denote $A_{n-1,n+1}$ The stress energy tensor corresponding to $\ue$ is
\begin{equation} \label{varifold}
V_\ve \coloneqq \frac{1}{2}|\nabla \ue|^2 \, T_x  
\end{equation}
where $T_x$ is given by
\begin{equation} \label{defT}
    T_x =
    \begin{cases}
        [I_{n+1} - 2 \nu_\ve \otimes \nu_\ve] \text{ if }|\nabla \ue (x)| \neq 0, \text{ here } \nu_\ve(x) = \frac{\nabla \ue (x)}{|\nabla \ue (x)|} \quad  \\
        0  \quad \text{if }|\nabla \ue (x)|=0
    \end{cases}
\end{equation}
Observe that 
$$
tr(T_x)=n-1 \text{ and }  -|v|^2 \leq \langle T_x(v),v \rangle \leq |v|^2 \quad \text{for any } v
$$
therefore $T_x \in A_{n-1}$ whenever $|\nabla \ue| \neq 0$ and we have the following definition. 
\begin{defi}
    Let $u_{\ve}$ satisfy (\ref{sec4eqn}). Then, there is an associated $(n-1)$ generalized varifold given by
    \begin{equation} \label{eqn:gvarifold}
    \langle V_\ve,f \rangle = \int_{|\nabla \ue|\neq 0} \frac{1}{2}|\nabla \ue|^2 f(T_x) \,dx \quad \text{for all } f\in C^0(A_{n-1}) 
\end{equation}
\end{defi}

There is also the related notion of convergece. 
\begin{defi}
    We say that a sequence of generalized varifolds $\{V_i\}$ converges to a generalized varifold $V$ if $\{V_i\}$ converge as Radon measures in $U \times A_{n-1}$ to $V$.
\end{defi}

\subsection{Rectifiablity of the limiting varifold} We will prove that there is a (classical) varifold with mass measure $\mu_{\Sigma}$ supported on $\Sigma$, and that it is stationary and rectifiable. Let $V_\ve$ be the generalized varifold associated to $\ue$ by (\ref{eqn:gvarifold}). We will show that by Theorem \ref{thm:sing},  $V_\ve \rightharpoonup V = V_* + V_{\Sigma}$, and the mass measure of $V_\Sigma$ is $\mu_{\Sigma}$. Further as a consequence of Proposition \ref{vanishing} we will show that $V$ is stationary. We are also able to show that $V_*$ and $V_{\Sigma}$ are stationary. As we already established a lower and upper bounds on density $\theta(x)$ for $\mathcal{H}^{n-1}$ a.e. $x \in \Sigma$ (Theorem \ref{thm:sing}(3)),  we will invoke a result in \cite{as} to conclud that $V_{\Sigma}$ is actually a stationary, rectifiable, $(n-1)$-varifold and in particular $\Sigma$ is a $(n-1)$-rectifiable set. 

\begin{theo}
Let $u_i$ satisfy (\ref{sec4eqn}) and $V_i$ be the generalized varifolds associated to it. Then as $i \to \infty$, we have the following,
\begin{enumerate}
    \item After passing to a subsequence, $V_i \rightharpoonup V= V_* + V_{\Sigma}$, here $V_*$ is the generalized varifold associated to $u_*$, and $V_{\Sigma}$ is a generalized varifold supported on $\Sigma \times A_{n-1}$ with mass measure $\mu_{\Sigma}$.
    \item $V$, $V_*$ and $V_{\Sigma}$ are stationary generalized varifolds. 
    \item And further $V_{\Sigma}$ can also be associated to a stationary, rectifiable varifold, with density $\theta(x)$ and supported on a $(n-1)$ rectifiable set $\Sigma$, i.e $V_{\Sigma} = V(\Sigma, \theta)$.
\end{enumerate}
\end{theo}

\begin{proof}
We first compute the first variation of $\delta V_i$. For any vector field $X\in C^1(U\cup \partial^0U)$ with $X_{n+1}=0$, the first variation is
\begin{align*}
    \langle \delta V_i, X \rangle
    &= \int_{U} \langle V_i , DX \rangle  dx  \\
    &= \int_U \frac{1}{2}|\nabla u_i|^2\langle T_x, DX \rangle \,dx \\
    &= \int_U \frac{1}{2}|\nabla u_i|^2\D X - 2DX(\nabla u_i, \nabla u_i)  \,dx \text{ by equation (\ref{defT})} \\
    &= \int_{\partial^0 U} \frac{-W(u_i)}{\ve_i} \D_{\RR^n} X \mh \quad \text{by equation (\ref{eqn:fv})} \\
    &\leq  \sup_{\partial^0 U}|\D_{\RR^n} X| \cdot \int_{\partial^0 U} \frac{W(u_i)}{\ve_i}  \mh  \quad \text{as } W \geq 0 \\
    &\leq C_{X} \ve_i \quad \text{by Theorem \ref{vanishing}}
\end{align*}
As $\mu_{V_i}(U) < E_0$, after passing to a subsequence there is a generalized varifold $V$ such that $V_i \rightharpoonup V$. In particular we have
\begin{align}
    \langle \delta V, X \rangle 
    &= \lim_{i \to \infty} \langle \delta V_i, X \rangle \\
    &=0   \label{Tsta}
\end{align}
Hence, $V$ is stationary. Further note that,
\begin{align*}
    \langle \delta V, X \rangle 
    &= \lim_{i\to \infty}\int_{U} \langle V_i , DX \rangle  dx \\
    &=\lim_{i \to \infty} \frac{1}{2} \int_{U}|\nabla \ui|^2 \langle I_{n+1} - 2 \nu_i \otimes \nu_i , DX \rangle \,dx  \\
    &=\lim_{i \to \infty} \frac{1}{2} \int_{U}|\nabla \ui|^2 \D X - 2DX\langle\nabla \ui,\nabla \ui\rangle \,dx  \\
\end{align*}
Note that by (\ref{measdecom}) we get
$$
\lim_{i \to \infty} \int_{U}|\nabla \ui|^2 \D X \,dx = \int_U |\nabla u_*|^2 \D X \,dx + \int_{\Sigma} \D_{\Sigma} X \,d\mu_{\Sigma}
$$
and because $\ui \rightharpoonup u_*$ weakly in $H^1(U)$ we get
$$
\lim_{i \to \infty}  \int_{U} DX\langle\nabla \ui,\nabla \ui\rangle \,dx = \int_U DX\langle\nabla u_*, \nabla u_*\rangle \,dx
$$
Combining these we resume the calculation
\begin{align*} 
    \langle \delta V, X \rangle  
    &= \frac{1}{2} \int_{U}|\nabla u_*|^2 \D X \,dx - 2DX\langle\nabla u_*,\nabla u_*\rangle \,dx+ \int_{\Sigma} \D_{\Sigma} X \,d\mu_{\Sigma}  \\
    &= \frac{1}{2} \int_{U}|\nabla u_*|^2 \langle I_{n+1} - 2 \nu_* \otimes \nu_* , DX \rangle \,dx + \int_{\Sigma} \D_{\Sigma} X \,d\mu_{\Sigma}  \\
    &=  \int_{U} \frac{1}{2}|\nabla u_*|^2 \langle T_x , DX \rangle \,dx + \int_{\Sigma} \D_{\Sigma} X \,d\mu_{\Sigma}                  \\
    &= \int_{U} \frac{1}{2}|\nabla u_*|^2 \langle T_x , DX \rangle \,dx + \int_{\Sigma} \theta(x) \D_{\Sigma} X \,d\mathcal{H}^{n-1} \text{ by Theorem \ref{thm:sing}}(3) \\
    &= \langle \delta V_*,X \rangle + \langle \delta V_{\Sigma},X \rangle 
\end{align*}
 Combined with equation (\ref{Tsta}), we get
$$\langle \delta V_{\Sigma},X \rangle  = -\langle \delta V_*,X \rangle$$
 We now show that $V_{\Sigma}$ is stationary. By the above, it is equivalent to showing that $\langle \delta V_*, X \rangle = 0$.
\begin{align*}
    \langle \delta V_{\Sigma},X \rangle 
    &= -\langle \delta V_*, X \rangle \\
    &= -\int_{U} \bigg( \frac{|\nabla u_*|^2}{2} \D X - DX\langle\nabla u_*,\nabla u_*\rangle \bigg)\,dx  \\
    &=-\int_{\partial^0U} (\nabla u_* \cdot X)\frac{\partial u_*}{\partial \nu} \mh \quad \text{by 2nd term in equation (\ref{ibp})} \\
    &= -\int_{\partial^0U \backslash \Sigma}\nabla(\pm 1)\cdot X \pdv{u_*}{\nu} \mh \text{ Proposition \ref{onu*} and } \mathcal{H}^n(\Sigma)=0   \\
    &= 0 
\end{align*}

 Hence, $V_{\Sigma}$ is stationary. We already have lower density bound $\mathcal{H}^{n-1}$ a.e. on $\Sigma$ for the mass measure $\mu_{\Sigma}$,  then by [Theorem 3.8(c)] in \cite{as}, $V_{\Sigma}$ is a stationary, $(n-1)$ rectifiable varifold with mass measure $\mu_{\Sigma}$, supported on $\Sigma$ . In particular, the set $\Sigma$ is $(n-1)$ rectifiable.
\end{proof}


\end{document}